\documentclass[12pt, reqno]{amsart}
\usepackage{amsmath, amsthm, amscd, amsfonts, amssymb, graphicx, color}
\usepackage[bookmarksnumbered, colorlinks, plainpages]{hyperref}
\hypersetup{colorlinks=true,linkcolor=red, anchorcolor=green, citecolor=cyan, urlcolor=red, filecolor=magenta, pdftoolbar=true}

\newtheorem{theorem}{Theorem}[section]
\newtheorem{lemma}[theorem]{Lemma}
\newtheorem{proposition}[theorem]{Proposition}
\newtheorem{corollary}[theorem]{Corollary}
\theoremstyle{definition}
\newtheorem{definition}[theorem]{Definition}
\newtheorem{example}[theorem]{Example}

\newtheorem{remark}[theorem]{Remark}
\numberwithin{equation}{section}
\begin{document}
	
	\setcounter{page}{1}
	
	\title{Operator frame for $Hom_{\mathcal{A}}^{\ast}(\mathcal{X})$}
	
	\author[R. Eljazzar, M. Rossafi, C. Park]{Roumaissae Eljazzar$^{1}$, Mohamed Rossafi$^2$ {and} Choonkil Park$^{3*}$}
	
\address{$^{1}$Department of Mathematics, Faculty of Sciences, University of Ibn Tofail, Kenitra, Morocco}
	\email{\textcolor[rgb]{0.00,0.00,0.84}{roumaissae.eljazzar@uit.ac.ma}}
	
\address{$^{2}$LASMA Laboratory Department of Mathematics Faculty of Sciences, Dhar El Mahraz University Sidi Mohamed Ben Abdellah, Fes, Morocco}
\email{\textcolor[rgb]{0.00,0.00,0.84}{rossafimohamed@gmail.com; mohamed.rossafi@usmba.ac.ma}}

\address{$^{3}$Research Institute for Natural Sciences, Hanyang University, Seoul 04763,  Korea}
\email{\textcolor[rgb]{0.00,0.00,0.84}{baak@hanyang.ac.kr}}

	
	\keywords{Frame;  Operator frame;  Pro-$C^{\ast}$-algebra;  Hilbert pro-$C^{\ast}$-modules;  Tensor product.\\ $^{*}$Corresponding author: Choonkil Park (email: baak@hanyang.ac.kr, fax: +82-2-2281-0019).}

	\setcounter{page}{1}
	
	\title[Operator frame for $Hom_{\mathcal{A}}^{\ast}(\mathcal{X})$]{Operator frame for $Hom_{\mathcal{A}}^{\ast}(\mathcal{X})$}
	
	 \maketitle
\begin{abstract}
	 The concept of operator frame can be considered as a generalization of frame. Firstly, we  introduce the notion of operator frame for the set of all adjointable operators $Hom_{\mathcal{A}}^{\ast}(\mathcal{X})$ on a Hilbert pro-$C^{\ast}$-module $\mathcal{X}.$ The analysis operator, the synthesis operator and the frame operator are presented. Secondly,  we study the stability of operator frame under small
	 perturbations. We also study the tensor product of operator frame  for Hilbert pro-$C^{\ast}$-modules. Finally, we establish its dual and some properties.
\end{abstract}

\section{Introduction}

The theory of frames for Hilbert spaces was first introduced in 1952 by Duffin and Schaefer \cite{Duf} for the study of nonharmonic series. Later,  Frank and Larson \cite{M.Frank, Frank} introduced the notion of standard frame of multipliers  for a Hilbert $C^{\ast}$-module and proved some of their basic properties.
In 2008, Joita \cite{Joita} proposed frames of multipliers in Hilbert pro-$C^{\ast}$-modules and showed  that many properties of frames in Hilbert $C^{\ast}$-modules are valid for frames of multipliers in Hilbert modules over pro-$C^{\ast}$-algebras.

Operator frames for $B(\mathcal{H})$ is a new notion of frames, which was  introduced  by  Li and Cao \cite{Cio} and was more generalized by Rossafi in \cite{Rossafi}. For more information and properties on frames, see \cite{bus,   nik,  san}. 

The aim of this work is to introduce the notion of operator frame for the space $Hom_{\mathcal{A}}^{\ast}(\mathcal{X})$ of all adjointable operators on a Hilbert pro-$C^{\ast}$-module for $\mathcal{X}$. We begin by recalling  some fundamental definitions
and notations of Hilbert pro-$C^{\ast}$-modules. In Section 2, we give the definition of operator frame and some properties. In Section 3, we study the stabilty of operator frame. In Section 4, we investigate the tensor product of Hilbert pro-$C^{\ast}$-modules and  we show that the  tensor product of operator frames for Hilbert pro-$C^{\ast}$-modules $\mathcal{X}$ and $\mathcal{Y}$ is an operator frame for $\mathcal{X} \otimes \mathcal{Y}$. Finally, the dual of operator frame and some properties are discussed.

\section{Preliminaries}

Recall that a pro-$C^{\ast}$-algebra is a complete Hausdorff complex topological $\ast$-algebra $\mathcal{A}$ whose topology is determined by its continuous $C^{\ast}$-seminorms in the sense that a net $\{a_{\alpha}\}$ converges to $0$ if 	and only if $p(a_{\alpha})$ converges to $0$ for all continuous $C^{\ast}$-seminorms $p$ on $\mathcal{A}$ (see \cite{Inoue,Lance,Philips}), and we have\\
\begin{enumerate}
	\item[(1)] $p(ab) \leq p(a)p(b)$,
	\item[(2)] $p(a^{\ast}a)=p(a)^{2}$
\end{enumerate}
for all  $a , b \in \mathcal{A}$.

If the topology of pro-$C^{*}$-algebra is determined by only countably many $C^{*}$-seminorms, then it is called
a $\sigma$-$C^{*}$-algebra.\\
We denote by $sp(a)$ the spectrum of $a$ such that $sp(a)=\left\{\lambda \in \mathbb{C}: \lambda 1_{\mathcal{A}}-a\right.$ is not invertible$\}$ for each  $a \in \mathcal{A}$,  where $\mathcal{A}$ is a unital pro-$C^{*}$-algebra with unit $1_{\mathcal{A}}$.\\
The set of all continuous $C^{\ast}$-seminorms on $\mathcal{A}$ is denoted by $S(\mathcal{A})$.
If $\mathcal{A}^{+}$ denotes the set of all positive elements of $\mathcal{A}$, then $\mathcal{A}^{+}$ is a closed convex $C^{*}$-seminorms on $\mathcal{A}.$ 
\begin{example}
	Every $C^{*}$-algebra is a pro-$C^{*}$-algebra.
	\end{example}

\begin{proposition} 
	Let $\mathcal{A}$ be a unital pro-$C^{*}$-algebra with unit  $1_{\mathcal{A}}.$ Then for any $p \in S(\mathcal{A}),$ we have
	\begin{enumerate}
		\item[(1)] $p(a)= p(a^{*})$ for all $a \in A$;
		\item[(2)] $p \left(1_{\mathcal{A}}\right)=1$;
		\item[(3)] If $a, b \in \mathcal{A}^{+}$ and $a \leq b$, then $ p(a) \leq p(b)$;
		\item[(4)] If $1_{\mathcal{A}} \leq b$, then $b$ is invertible and $b^{-1} \leq 1_{\mathcal{A}}$;
		\item[(5)] If $a, b \in \mathcal{A}^{+}$ are invertible and $0 \leq a \leq b$, then $0 \leq b^{-1} \leq a^{-1}$;
		\item[(6)] If $a, b, c \in \mathcal{A}$ and $a \leq b$, then $c^{*} a c \leq c^{*} b c$;
		\item[(7)] If $a, b \in \mathcal{A}^{+}$ and $a^{2} \leq b^{2}$, then $0 \leq a \leq b$.
	\end{enumerate}
\end{proposition}

\begin{definition}\cite{Philips}
A pre-Hilbert module over a pro-$C^{\ast}$-algebra $\mathcal{A}$ is a complex vector space $E$ which is also
a left $\mathcal{A}$-module compatible with the complex algebra structure, equipped with an $\mathcal{A}$-valued inner product $\langle .,.\rangle$ $E \times E \rightarrow \mathcal{A}$ which is  $\mathcal{A}$-linear in its first variable and satisfies the following conditions:
\begin{enumerate}
	\item[1)]
	$\langle \xi, \eta\rangle^{*}=\langle \eta, \xi\rangle 
	$;
	\item[2)]$\langle \xi, \xi\rangle \geq 0$;

	\item[3)] $\langle \xi, \xi\rangle=0 $  if and only if  $\xi=0$
\end{enumerate}
for all $\xi, \eta \in E .$ We say $E$ is a Hilbert $\mathcal{A}$-module (or Hilbert pro-$C^{\ast}$-module over $\mathcal{A}$). Assume that  $E$ is complete with respect to the topology determined by the family of seminorms
$$
\bar{p}_{E}(\xi)=\sqrt{p(\langle \xi, \xi\rangle)} \quad \xi \in E, p \in S(\mathcal{A}).
$$
\end{definition}
Let $\mathcal{A}$ be a pro-$C^{\ast}$-algebra and let $\mathcal{X}$ and $\mathcal{Y}$ be Hilbert $\mathcal{A}$-modules and assume that $I$ and $J$ are countable index sets.
A bounded $\mathcal{A}$-module map from
$\mathcal{X}$ to $\mathcal{Y}$ is called an operator from $\mathcal{X}$ to $\mathcal{Y}$. We denote the set of all operators from $\mathcal{X}$ to $\mathcal{Y}$ by  $Hom_{\mathcal{A}}(\mathcal{X}, \mathcal{Y})$.

\begin{definition}
An $ \mathcal{A}$-module map $T: \mathcal{X} \longrightarrow \mathcal{Y}$  is adjointable if there is a map $T^{\ast}: \mathcal{Y} \rightarrow \mathcal{X}$ such that $\langle T \xi, \eta\rangle=\left\langle \xi, T^{\ast} \eta\right\rangle$ for all $\xi \in \mathcal{X}, \eta \in \mathcal{Y}$, and is called bounded if for all $p \in S(\mathcal{A})$, there is $M_{p}>0$ such that $\bar{p}_{\mathcal{Y}}(T \xi) \leq M_{p} \bar{p}_{\mathcal{X}}(\xi)$ for all $\xi \in \mathcal{X}$.

We denote by $Hom_{\mathcal{A}}^{\ast}(\mathcal{X}, \mathcal{Y})$ the set of all adjointable operators from $\mathcal{X}$ to $\mathcal{Y}$ and let $Hom_{\mathcal{A}}^{\ast}(\mathcal{X})=Hom_{\mathcal{A}}^{\ast}(\mathcal{X}, \mathcal{X})$	 
\end{definition}

\begin{definition}
	Let $\mathcal{A}$ be a pro-$C^{\ast}$-algebra and $\mathcal{X}, \mathcal{Y}$ be two Hilbert $\mathcal{A}$-modules. An  operator $T: \mathcal{X} \rightarrow \mathcal{Y}$ is called uniformly bounded below if there exists $C>0$ such that for each $p \in S(\mathcal{A})$,
	\begin{equation*}
	\bar{p}_{\mathcal{Y}}(T \xi) \leqslant C \bar{p}_{\mathcal{X}}(\xi), \quad \text { for all } \xi \in \mathcal{X} 
	\end{equation*}
	and  is called uniformly bounded  above if there exists $C^{\prime}>0$ such that for each $p \in S(\mathcal{A})$,
	\begin{equation*}
	\bar{p}_{\mathcal{Y}}(T \xi) \geqslant C^{\prime}  \bar{p}_{\mathcal{X}}(\xi), \quad \text { for all } \xi \in \mathcal{X}.
	\end{equation*}
	Let	
	\begin{equation*}
	\|T\|_{\infty}=\inf \{M: M \text { is an upper bound for } T\},
	\end{equation*}
	\begin{equation*}
	\hat{p}_{\mathcal{Y}}(T)=\sup \left\{\bar{p}_{\mathcal{Y}}(T(x)): \xi \in \mathcal{X}, \quad \bar{p}_{\mathcal{X}}(\xi) \leqslant 1\right\}.
	\end{equation*}
	It is  clear to see that  $\hat{p}(T) \leqslant\|T\|_{\infty}$ for all $p \in S(\mathcal{A})$.	
\end{definition}

\begin{proposition} \cite{Azhini} \label{Prop2.6}
	Let $\mathcal{X}$ be a Hilbert module over pro-$C^{*}$-algebra $\mathcal{A}$ and $T$ be an invertible element in $Hom_{\mathcal{A}}^{\ast}(\mathcal{X})$, which is  uniformly bounded. Then for all $\xi \in \mathcal{X}$,
	$$
	\left\|T^{-1}\right\|_{\infty}^{-2}\langle \xi, \xi\rangle \leq\langle T \xi, T \xi \rangle \leq\|T\|_{\infty}^{2}\langle \xi, \xi\rangle
	.$$
\end{proposition}

\begin{definition}
	A sequence $\{x_{i}\}_{i \in J}$ of elements in a Hilbert $\mathcal{A}$-module $\mathcal{X}$ is said to be a frame if there are real constants $C, D>0$ such that:
	\begin{equation}\label{equa3}
	C\langle \xi, \xi\rangle \leq \sum_{i \in J}\left\langle \xi, x_{i}\right\rangle\left\langle x_{i}, \xi\right\rangle \leq D\langle \xi, \xi\rangle
\end{equation}
	for all $\xi \in \mathcal{X}$. The constants $C$ and $D$ are called lower and upper frame bounds for the frame, respectively. If $C=D=\lambda$, the frame is called a $\lambda$-tight frame. If $C=D=1$, the frame is called a Parseval frame. We say that $\{x_{i}\}_{ i \in J}$ is a frame sequence if it is a frame for the closure of $\mathcal{A}$-linear hull of $\{x_{i}\}_{ i \in J}$.  If, in (\ref{equa3}), we only require to have the upper bound, then $\{x_{i}\}_{i \in J}$ is called a Bessel sequence with Bessel bound $D$.
\end{definition}

\begin{definition} \cite{Azhini}
	Let $\{v_{i}\}_{i \in J}$ be a family of weights in $A$, i.e., each $v_{i}$ is a positive invertible element from the center of $\mathcal{A}$, and let $\left\{M_{i}: i \in J \right\}$ be a family of orthogonally complemented submodules of $\mathcal{X} .$ Then $\left\{\left(M_{i}, v_{i}\right): i \in J\right\}$ is a fusion frame if there exist real constants $C, D>$ 0 such that:
	\begin{equation}\label{equa4}
	C\langle \xi, \xi\rangle \leq \sum_{i \in J} v_{i}^{2}\left\langle P_{M_{i}}(\xi), P_{M_{i}}(\xi)\right\rangle \leq D\langle \xi, \xi\rangle
	\end{equation}
	for all $\xi \in \mathcal{X}$. We call $C$ and $D$ the lower and upper bounds of the fusion frame. If $C=D=\lambda$, then  the family $\left\{\left(M_{i}, v_{i}\right): i \in J \right\}$ is called a $\lambda$-tight fusion frame and if $C=D=1$, then  it is called a Parseval fusion frame. If, in (\ref{equa4}), we only have the upper bound, then $\left\{\left(M_{i}, v_{i}\right): i \in J \right\}$ is called a Bessel fusion sequence with Bessel bound $D$.
\end{definition}
 See \cite{kna, lil} for more infornation on fusion frames and applications. 

\begin{lemma}\cite{Alizadeh}\label{lem1.9}
		 Let $\mathcal{A}$ be a locally $C^{\ast}$-algebra, $\mathcal{X}$ be  a Hilbert $\mathcal{A}$-module and $T $ be a  bounded operator in  $Hom_{\mathcal{A}}^{\ast}(\mathcal{X}) $ such that $T^{*}=T .$ The following statements are mutually equivalent:
		 \begin{enumerate}	 
	\item $T$ is surjective.
	\item  There are $m, M>0$ such that $m \bar{p}_{\mathcal{X}}(\xi) \leq \bar{p}_{\mathcal{X}}(T \xi) \leq M \bar{p}_{\mathcal{X}}(\xi)$ for all
	$p \in S(\mathcal{A})$ and $\xi \in \mathcal{X}$.
	\item  There are $m^{\prime}, M^{\prime}>0$ such that $m^{\prime}\langle \xi, \xi \rangle \leq\langle T \xi, T \xi \rangle \leq M^{\prime}\langle \xi, \xi \rangle$ for all
	$\xi \in \mathcal{X}.$
\end{enumerate}
	\end{lemma}
Similar to $C^{\ast}$-algebra,  each  $\ast$-homomorphism between two pro-$C^{\ast}$-algebra is increasing.

\begin{lemma}\label{2.7} 
	If $\varphi:\mathcal{A}\rightarrow\mathcal{B}$ is a $\ast$-homomorphism between pro-$\mathcal{C}^{\ast}$-algebras, then $\varphi$ is increasing, that is, if $a\leq b$, then $\varphi(a)\leq\varphi(b)$.
\end{lemma}

\begin{proof}
The proof is easy.
\end{proof}

\section{Operator frame for $ Hom_{\mathcal{A}}^{\ast}(\mathcal{X})$}

\begin{definition}
	A family of adjointable operators $\{T_{i}\}_{i\in J}$ on a Hilbert $\mathcal{A}$-module $\mathcal{X}$ over a unital pro-$C^{\ast}$-algebra is said to be an operator frame for $Hom_{\mathcal{A}}^{\ast}(\mathcal{X})$ if there exist   constants $A, B > 0$ such that 
	\begin{equation}\label{eq3}
		A\langle \xi,\xi\rangle\leq\sum_{i\in J}\langle T_{i}\xi,T_{i}\xi\rangle\leq B\langle \xi,\xi\rangle, \forall \xi\in\mathcal{X}.
	\end{equation}
	The numbers $A$ and $B$ are called lower and upper bounds of the operator frame, respectively. If $A=B=\lambda$, then the operator frame is $\lambda$-tight. If $A = B = 1$, then it is called a normalized tight operator frame or a Parseval operator frame. If only upper inequality of \eqref{eq3} holds, then $\{T_{i}\}_{i\in J}$ is called an operator Bessel sequence for $Hom_{\mathcal{A}}^{\ast}(\mathcal{X})$.
	
	If the sum in the middle of \eqref{eq3} is convergent in norm, then the operator frame is called standard. 
\end{definition}

\begin{example}
	Let $\mathcal{A}$ be a Hilbert pro-$C^{\ast}$-module over itself with the inner product $\langle a,b\rangle=ab^{\ast}$.
	Let $\{x_{i}\}_{i\in J}$ be a frame for $\mathcal{A}$ with bounds $A$ and $B$, respectively. For each $i\in J$, we define $T_{i}:\mathcal{A}\to\mathcal{A}$ by $T_{i}\xi=\langle \xi,x_{i}\rangle,\;\; \forall \xi\in\mathcal{A}$. Then $T_{i}$ is adjointable and $T_{i}^{\ast}a=ax_{i}$ for all $a\in\mathcal{A}$. We have 
	\begin{equation*}
		A\langle \xi,\xi\rangle\leq\sum_{i\in J}\langle \xi,x_{i}\rangle\langle x_{i},\xi\rangle\leq B\langle \xi,\xi\rangle, \forall \xi\in\mathcal{A}
	\end{equation*}
	and so 
	\begin{equation*}
		A\langle \xi,\xi\rangle\leq\sum_{i\in J}\langle T_{i}\xi,T_{i}\xi\rangle\leq B\langle \xi,\xi\rangle, \forall \xi\in\mathcal{A}.
	\end{equation*}
Thus  $\{T_{i}\}_{i\in J}$ is an operator frame in $\mathcal{A}$ with bounds $A$ and $B$, respectively.
\end{example}

\begin{example}
	Let $\{W_{i}\}_{i\in J}$ be a frame of submodules with respect to $ \{v_{i}\}_{i\in J} $ for $\mathcal{X}$. Put $T_{i}=v_{i}\pi_{W_{i}}, \forall i\in J$. Then we get a sequence of operators $\{T_{i}\}_{i\in J}$ and  there exist constants $A, B >0$ such that
	\begin{displaymath}
		A\langle \xi,\xi\rangle\leq\sum_{i\in J}v_{i}^{2}\langle P_{W_{i}}\xi,P_{W_{i}}\xi\rangle\leq B\langle \xi,\xi\rangle, \forall \xi\in\mathcal{X}.
	\end{displaymath}
	So we have
	\begin{displaymath}
		A\langle \xi,\xi\rangle\leq\sum_{i\in J}\langle T_{i}\xi,T_{i}\xi\rangle\leq B\langle \xi,\xi\rangle, \forall \xi\in\mathcal{X}.
	\end{displaymath}
	Thus the sequence $\{T_{i}\}_{i\in J}$ becomes an operator frame for $\mathcal{X}$.
\end{example}

With this example,  a frame of submodules can be viewed as a special case of operator frames.

\begin{theorem}
	Let $\{T_{i}\}_{i \in J}$ be an operator frame for $Hom_{\mathcal{A}}^{\ast}(\mathcal{X})$ with bounds $A$ and $B$. Let $Q$ be a bounded element in $Hom_{\mathcal{A}}(\mathcal{X})$ such that $Q= Q^{\ast}.$ If $Q$ is surjective, then $\{T_{i}Q\}_{i \in J}$ is an operator frame for $Hom_{\mathcal{A}}^{\ast}(\mathcal{X}).$
\end{theorem}

	\begin{proof}
	Since 	$Q$ is surjective,  by Lemma \ref{lem1.9},  there are $m, M>0$ such that $$m \langle \xi, \xi\rangle \leq\langle Q \xi, Q \xi \rangle \leq M \langle \xi, \xi\rangle \quad  \text{for all} \quad
		\xi \in \mathcal{X}.$$ 
		Then 
		\begin{displaymath}
		Am\langle \xi,\xi\rangle\leq\sum_{i\in J}\langle T_{i}Q\xi,T_{i}Q\xi\rangle\leq B M\langle \xi,\xi\rangle, \forall \xi\in\mathcal{X}.
		\end{displaymath}
		Hence $\{T_{i}Q\}_{i \in J}$ is an operator frame for $Hom_{\mathcal{A}}^{\ast}(\mathcal{X}).$		
	\end{proof}

Let $\{T_{i}\}_{i\in J}$ be an operator frame for $Hom_{\mathcal{A}}^{\ast}(\mathcal{X})$. Define an operator 
$$R:\mathcal{X}\rightarrow l^{2}(\mathcal{X})\;\;by\;\; Rx=\{T_{i}x\}_{i\in J}, \forall x\in\mathcal{X}.$$
The operator $R$ is called an  analysis operator of the operator frame $ J \{T_{i}\}_{i\in J} $.\\
The adjoint $R^{\ast}(\{x_{i}\}_{i\in J}):l^{2}(\mathcal{X})\rightarrow\mathcal{X}$ of the  analysis operator $R$  is defined by $$R^{\ast}(\{x_{i}\}_{i\in J})=\sum_{i\in J}T_{i}^{\ast}x_{i}, \forall\{x_{i}\}_{i\in J}\in l^{2}(\mathcal{X}).$$
The operator $R^{\ast}$ is called the synthesis operator of the operator frame $ \{T_{i}\}_{i\in J} $.\\
By composing $R$ and $R^{\ast}$, the frame operator $S:\mathcal{X}\rightarrow\mathcal{X}$ for the operator frame is given by $$	S(x)=R^{\ast}Rx=\sum_{i\in  J}T_{i}^{\ast}T_{i}x.$$

\begin{theorem}
	Let $\{T_{i}\}_{i\in J}$ be a family of adjointable operators  on a Hilbert $\mathcal{A}$-module $\mathcal{X}$. Then $\{T_{i}\}_{i\in J}$ is an operator frame for $Hom_{\mathcal{A}}^{\ast}(\mathcal{X})$ if and only if there exist positive constants $A, B > 0$ such that
	\begin{equation}\label{equa2}
		A\bar{p}_{\mathcal{X}}(\xi)^{2}\leq p(\sum_{i\in J}\langle T_{i}\xi,T_{i}\xi\rangle) \leq B \bar{p}_{\mathcal{X}}(\xi)^{2}, \forall \xi\in\mathcal{X}.
	\end{equation} 
\end{theorem}

\begin{proof}
	Suppose that $\{T_{i}\}_{i\in J}$ is an operator frame for $Hom_{\mathcal{A}}^{\ast}(\mathcal{X})$.  Combined with the definition of an operator frame, we know that \eqref{equa2} holds.

	Now suppose that \eqref{equa2} holds. We know that the frame operator $S$ is positive, self-adjoint and invertible, and hence $$\langle S^{\frac{1}{2}}x,S^{\frac{1}{2}}\xi\rangle=\langle S\xi,\xi\rangle=\sum_{i\in J}\langle T_{i}\xi,T_{i}\xi\rangle.$$
	So we have $A p(\langle \xi,\xi\rangle)\leq p(\langle S^{\frac{1}{2}}\xi,S^{\frac{1}{2}}\xi\rangle)\leq B p(\langle \xi,\xi\rangle)$ for all $\xi\in\mathcal{X}.$  Thus  $ A\bar{p}_{\mathcal{X}}(\xi)^{2}\leq \bar{p}_{\mathcal{Y}}( S^{\frac{1}{2}}\xi)^{2}\leq B \bar{p}_{\mathcal{X}}(\xi)^{2}.$  According to Proposition \ref{Prop2.6}, we have
	$$\|S^{- \frac{1}{2}}\|_{\infty}^{-2}\langle \xi, \xi\rangle\leq\langle S^{\frac{1}{2}}\xi,S^{\frac{1}{2}}\xi\rangle=\sum_{i\in J}\langle T_{i}\xi,T_{i}\xi\rangle\leq \|S^{ \frac{1}{2}}\|_{\infty}^{2}\langle \xi,\xi\rangle,$$
	which implies that $\{T_{i}\}_{i\in J}$ is an operator frame for $Hom_{\mathcal{A}}^{\ast}(\mathcal{X})$.
\end{proof}

\begin{theorem}
	Assume that $S$ is the frame operator of an operator frame $T=\{T_{i}\}_{i\in J}$ for $Hom_{\mathcal{A}}^{\ast}(\mathcal{X})$ with bounds $A$ and $B$. Then $S$ is positive, self-adjoint and invertible. Moreover, we have $AI\leq S \leq BI$ and the reconstruction formula
	\begin{equation}
		\xi=\sum_{i\in J}S^{-1}T_{i}^{\ast}T_{i}\xi, \forall \xi\in\mathcal{X}.
	\end{equation}  
\end{theorem}

\begin{proof}
	It is clear that $S$ is positive and self-adjoint. For any $\xi\in\mathcal{X}$, since $T=\{T_{i}\}_{i\in J}$ is an operator frame with bounds $A, B$, we have
	\begin{displaymath}
		\langle A\xi,\xi\rangle=A\langle \xi,\xi\rangle\leq\sum_{i\in J}\langle T_{i}\xi,T_{i}\xi\rangle=\langle S\xi,\xi\rangle\leq B\langle \xi,\xi\rangle=\langle B\xi,\xi\rangle.
	\end{displaymath}
	This shows that
	\begin{displaymath}
		AI\leq S \leq BI,
	\end{displaymath}
	which implies that $S$ is invertible. Further, for any $\xi\in\mathcal{X}$, we have
	\begin{displaymath}
		\xi=S_{T}^{-1}S_{T}\xi=S_{T}^{-1}\sum_{i\in J}T_{i}^{\ast}T_{i}\xi=\sum_{i\in  J}S_{T}^{-1}T_{i}^{\ast}T_{i}\xi.
	\end{displaymath}
This completes the proof.
\end{proof}

\begin{theorem}
	Let $(\mathcal{X},\mathcal{A},\langle.,.\rangle_{\mathcal{A}})$ and $(\mathcal{X},\mathcal{B},\langle.,.\rangle_{\mathcal{B}})$ be two Hilbert pro-$\mathcal{C^{\ast}}$-modules and $\varphi :\mathcal{A}\longrightarrow \mathcal{B}$ be a $\ast$-homomorphism and $\theta \in Hom_{\mathcal{A}}^{\ast}(\mathcal{X})$ be an invertible map such that both are uniformly bounded  and  $\langle \theta \xi,\theta \eta\rangle_{\mathcal{B}}=\varphi(\langle \xi, \eta\rangle_{\mathcal{A}})$ for all $\xi,\eta\in\mathcal{X}$. Suppose that $\{T_{i}\}_{i \in J}$ is an operator frame for $(\mathcal{X},\mathcal{A},\langle.,.\rangle_{\mathcal{A}})$ with frame operator $S_{\mathcal{A}} $ and lower and upper operator frame bounds $A$ and $B$,  respectively. Then $\{\theta T_{i}\}_{i}$ is an operator frame for $(\mathcal{X},\mathcal{B},\langle.,.\rangle_{\mathcal{B}})$ with frame operator $S_{\mathcal{B}} $ and lower and upper operator frame bounds $\|\theta^{-1}\|_{\infty}^{-2}A$ ,$\|\theta\|_{\infty}^{2}B$,  respectively, and $\langle S_{\mathcal{B}} \xi,\eta\rangle_{\mathcal{B}}=\varphi(\langle S_{\mathcal{A}}\xi, \eta\rangle_{\mathcal{A}})$.
\end{theorem}

\begin{proof} By the definition of operator frame,  we have
	$$A\langle \xi,\xi\rangle_{\mathcal{A}}\leq\sum_{i\in J}\langle T_{i}\xi,T_{i}\xi\rangle_{\mathcal{A}}\leq B\langle \xi,\xi\rangle_{\mathcal{A}} , \forall \xi\in\mathcal{X}.$$
	By Lemma \ref{2.7},  we have
	$$\varphi(A\langle \xi,\xi\rangle_{\mathcal{A}})\leq\varphi(\sum_{i\in J}\langle T_{i}\xi,T_{i}\xi\rangle_{\mathcal{A}})\leq\varphi( B\langle \xi,\xi\rangle_{\mathcal{A}}) , \forall \xi\in\mathcal{X}.$$
	By the definition of $\ast$-homomorphism,  we have
	$$A\varphi(\langle \xi,\xi\rangle_{\mathcal{A}}) \leq\sum_{i\in J}\varphi(\langle T_{i}\xi,T_{i}\xi\rangle_{\mathcal{A}})\leq B\varphi(\langle \xi,\xi\rangle_{\mathcal{A}}) , \forall \xi\in\mathcal{X}.$$
	By the relation betwen $\theta$ and $\varphi$,  we get
	$$A\langle \theta \xi,\theta \xi\rangle_{\mathcal{B}} \leq\sum_{i\in J}\langle \theta T_{i}\xi,\theta T_{i}\xi\rangle_{\mathcal{B}}\leq B\langle\theta \xi,\theta \xi\rangle_{\mathcal{B}} , \forall \xi\in\mathcal{X}.$$
	By Proposition \ref{Prop2.6},  we have
	$$\|\theta^{-1}\|_{\infty}^{-2}\langle \xi,\xi\rangle_{\mathcal{B}}\leq\langle \theta \xi,\theta \xi\rangle_{\mathcal{B}}\leq\|\theta\|_{\infty}^{2}\langle \xi,\xi\rangle_{\mathcal{B}}.$$
	Thus
	$$
	\|\theta^{-1}\|_{\infty}^{-2}A\langle  \xi, \xi\rangle_{\mathcal{B}}\leq\sum_{i\in J}\langle \theta T_{i}\xi,\theta T_{i}\xi\rangle_{\mathcal{B}}\leq\|\theta\|_{\infty}^{2}B\langle \xi,\xi\rangle_{\mathcal{B}} , \quad \forall \xi\in\mathcal{X}.
	$$
	On the other hand,  we have
	$$
	\aligned
	\varphi(\langle S_{\mathcal{A}}\xi, \eta\rangle_{\mathcal{A}})&=\varphi(\langle\sum_{i\in J}T_{i}^{\ast}T_{i}\xi,\eta\rangle_{\mathcal{A}})\\
	&=\sum_{i\in J}\varphi(\langle T_{i}\xi,T_{i}\eta\rangle_{\mathcal{A}})\\
	&=\sum_{i\in J}\langle\theta T_{i}\xi,\theta T_{i}\eta\rangle_{\mathcal{B}}
	\\
	&=\sum_{i\in J}\langle(\theta T_{i})^{\ast}(\theta T_{i})\xi,\eta\rangle_{\mathcal{B}}\\
	&=\langle\sum_{i\in J}(\theta T_{i})^{\ast}(\theta T_{i})\xi,\eta\rangle_{\mathcal{B}}\\
	&=\langle S_{\mathcal{B}} \xi,\eta\rangle_{\mathcal{B}}.
	\endaligned
$$
This  completes the proof.
\end{proof}

\section{Perturbation and stability of operator frames for $Hom_{\mathcal{A}}^{\ast}(\mathcal{X})$}

\begin{theorem}
  Let $\left\{T_{i}\right\}_{i \in J}$ be an operator frame for $Hom_{\mathcal{A}}^{\ast}(\mathcal{X})$ with bounds $A$ and
	B. If $\left\{R_{i}\right\}_{i \in J} \in Hom_{\mathcal{A}}^{\ast}(\mathcal{X})$ is an operator Bessel sequence with a bound $M<A$, then $\left\{T_{i} - R_{i}\right\}_{i \in \mathbb{J}}$ is an operator frame for $Hom_{\mathcal{A}}^{\ast}(\mathcal{X})$.
\end{theorem}

\begin{proof}
	For any $\xi \in \mathcal{X}$,  we have 
	$$
	\aligned
	p(\sum_{i \in J}\left\langle\left(T_{i}-R_{i}\right) \xi,\left(T_{i}-R_{i}\right) \xi \right\rangle)^{\frac{1}{2}}&=\bar{p}_{\mathcal{X}}\{ (T_{i}-R_{i})\xi \}_{i\in J}
	\\&\leq p(\sum_{i \in J}\langle T_{i} \xi,T_{i} \xi\rangle)^{\frac{1}{2}} + p(\sum_{i \in J}\langle R_{i} \xi,R_{i} \xi\rangle)^{\frac{1}{2}}\\
	&\leq \sqrt{B} \bar{p}_{\mathcal{X}}(\xi) + \sqrt{M} \bar{p}_{\mathcal{X}}(\xi)\\
	&\leq (\sqrt{B}+\sqrt{M})\bar{p}_{\mathcal{X}}(\xi).
	\endaligned
	$$
On the  other hand, 
	$$
	\aligned
	A p(\langle \xi,\xi \rangle)&\leq p(\sum_{i \in J}\langle T_{i} \xi,T_{i} \xi\rangle)\\&\leq p(\sum_{i \in J}\left\langle\left(T_{i}-R_{i}\right) \xi,\left(T_{i}-R_{i}\right) \xi \right\rangle) + p(\sum_{i \in J}\langle R_{i} \xi,R_{i} \xi\rangle)
	\\ &\leq p(\sum_{i \in J}\left\langle\left(T_{i}-R_{i}\right) \xi,\left(T_{i}-R_{i}\right) \xi \right\rangle) + M p(\langle \xi,\xi \rangle).
   \endaligned
	$$
	Thus 
	$$(A-M)p(\langle \xi,\xi \rangle) \leq p(\sum_{i \in J}\left\langle\left(T_{i}-R_{i}\right) \xi,\left(T_{i}-R_{i}\right) \xi \right\rangle). $$
	Consequently,
	$$(A-M)^{2} \bar{p}_{\mathcal{X}}(\xi)^{2} \leq p(\sum_{i \in J}\left\langle\left(T_{i}-R_{i}\right) \xi,\left(T_{i}-R_{i}\right) \xi \right\rangle) \leq(\sqrt{B}+\sqrt{M})^{2}\bar{p}_{\mathcal{X}}(\xi)^{2}. $$
	Hence  $\left\{T_{i} - R_{i}\right\}_{i \in \mathbb{J}}$ is an operator frame for $Hom_{\mathcal{A}}^{\ast}(\mathcal{X})$.
\end{proof}

\begin{theorem}
 Let $\{T_{i}\}_{i \in J}$ be an operator frame for $Hom_{\mathcal{A}}^{*}(\mathcal{X})$ with bound $A$ and $B$ and $\left\{R_{i}\right\}_{i \in J} \in Hom_{\mathcal{A}}^{*}(\mathcal{X}).$ Then the following statements are equivalent:
 \begin{enumerate}
 
	\item[(i)] $\left\{R_{i}\right\}_{i \in J}$ is an operator frame for $Hom_{\mathcal{A}}^{*}(\mathcal{X}).$
   \item[(ii)]  There exists a constant $M>0$ such that for all $x \in \mathcal{X}$, we have
	\begin{eqnarray}\label{Ineq4.1}
&& p(\sum_{i \in J}\left\langle\left(T_{i}-R_{i}\right) \xi,\left(T_{i}-R_{i}\right) \xi\right\rangle) \\ && \nonumber \qquad \leq M \min \left(p(\sum_{i \in J}\left\langle T_{i} \xi, T_{i} \xi\right\rangle),p(\sum_{i \in J}\left\langle R_{i} \xi, R_{i} \xi\right\rangle)\right).
	\end{eqnarray}	
\end{enumerate}
\end{theorem}

\begin{proof}
	First, suppose that $\left\{R_{i}\right\}_{i \in J} \in Hom_{\mathcal{A}}^{*}(\mathcal{X})$ is an operator frame with bounds $C$ and $D$. Then for any $\xi \in \mathcal{X}$, we have 
	$$
	\aligned
	p(\sum_{i \in J}\left\langle\left(T_{i}-R_{i}\right) \xi,\left(T_{i}-R_{i}\right) \xi \right\rangle)^{\frac{1}{2}}&=\bar{p}_{\mathcal{X}}\{ (T_{i}-R_{i})\xi \}_{i\in J}
	\\&\leq p(\sum_{i \in J}\langle T_{i} \xi,T_{i} \xi\rangle)^{\frac{1}{2}} + p(\sum_{i \in J}\langle R_{i} \xi,R_{i} \xi\rangle)^{\frac{1}{2}}\\
	&\leq  p(\sum_{i \in J}\langle T_{i} \xi,T_{i} \xi\rangle)^{\frac{1}{2}} + \sqrt{D} \bar{p}_{\mathcal{X}}(\xi)\\&\leq  p(\sum_{i \in J}\langle T_{i} \xi,T_{i} \xi\rangle)^{\frac{1}{2}} + \sqrt{\frac{D}{A}} p(\sum_{i \in J}\langle T_{i} \xi,T_{i} \xi\rangle)^{\frac{1}{2}}\\
	&\leq (1+\sqrt{\frac{D}{A}})p(\sum_{i \in J}\langle T_{i} \xi,T_{i} \xi\rangle)^{\frac{1}{2}}.
	\endaligned
	$$

	Similarly, we can obtain 
	$$ 	p(\sum_{i \in J}\left\langle\left(T_{i}-R_{i}\right) \xi,\left(T_{i}-R_{i}\right) \xi \right\rangle)^{\frac{1}{2}} \leq (1+\sqrt{\frac{B}{C}})p(\sum_{i \in J}\langle R_{i} \xi,R_{i} \xi\rangle)^{\frac{1}{2}}. $$
	Let $M=\min \left\{1+\sqrt{\frac{D}{A}}, 1+\sqrt{\frac{B}{C}}\right\}$. Then (\ref{Ineq4.1}) holds.
	
	Conversely, suppose that (\ref{Ineq4.1}) holds. For every $\xi \in \mathcal{X}$, we have
	$$
	\aligned
		\sqrt{A}\bar{p}_{\mathcal{X}}(\xi)
		&\leq p(\sum_{i \in J}\langle T_{i} \xi,T_{i} \xi\rangle)^{\frac{1}{2}} =\bar{p}_{\mathcal{X}}\{T_{i} \xi\}_{i \in J} \\ 
		&\leq \bar{p}_{\mathcal{X}}\{ (T_{i}-R_{i})\xi \}_{i\in J}+\bar{p}_{\mathcal{X}}\{R_{i} \xi\}_{i \in J}\\ 
		&\leq p(\sum_{i \in J}\left\langle\left(T_{i}-R_{i}\right) \xi,\left(T_{i}-R_{i}\right) \xi \right\rangle)^{\frac{1}{2}} + p(\sum_{i \in J}\langle R_{i} \xi,R_{i} \xi\rangle)^{\frac{1}{2}}\\
		&\leq \sqrt{M}p(\sum_{i \in J}\langle R_{i} \xi,R_{i} \xi\rangle)^{\frac{1}{2}}+ p(\sum_{i \in J}\langle R_{i} \xi,R_{i} \xi\rangle)^{\frac{1}{2}} \\
		&=(\sqrt{M}+1)p(\sum_{i \in J}\langle R_{i} \xi,R_{i} \xi\rangle)^{\frac{1}{2}}. 		
	\endaligned
	$$
	Also we have
	$$
	\aligned
	& p(\sum_{i \in J}\langle R_{i} \xi,R_{i} \xi\rangle)^{\frac{1}{2}} =\bar{p}_{\mathcal{X}}\{R_{i} \xi\}_{i \in J}
	\\ &\leq \bar{p}_{\mathcal{X}}\{ (T_{i}-R_{i})\xi \}_{i\in J}+\bar{p}_{\mathcal{X}}\{T_{i} \xi\}_{i \in J}\\
	&\leq p(\sum_{i \in J}\left\langle\left(T_{i}-R_{i}\right) \xi,\left(T_{i}-R_{i}\right) \xi \right\rangle)^{\frac{1}{2}} + p(\sum_{i \in J}\langle T_{i} \xi,T_{i} \xi\rangle)^{\frac{1}{2}}\\
	&\leq \sqrt{M}p(\sum_{i \in J}\langle T_{i} \xi,T_{i} \xi\rangle)^{\frac{1}{2}}+ p(\sum_{i \in J}\langle T_{i} \xi,T_{i} \xi\rangle)^{\frac{1}{2}} \\
	&=(\sqrt{M}+1)p(\sum_{i \in J}\langle T_{i} \xi,T_{i} \xi\rangle)^{\frac{1}{2}}.	
	\endaligned
	$$
	Thus 
	$$
	\frac{A}{(\sqrt{M}+1)^{2}}\bar{p}_{\mathcal{X}}( \xi)^{2}\leq p(\sum_{i \in \mathbb{J}}\left\langle R_{i} x, R_{i} x\right\rangle) \leq B(\sqrt{M}+1)^{2}\bar{p}_{\mathcal{X}}( \xi)^{2}. 								
	$$
So  $\left\{R_{i}\right\}_{i \in J}$ is an operator frame for $Hom_{\mathcal{A}}^{*}(\mathcal{X}).$
\end{proof}

\section{Tensor product}

The minimal or injective tensor product of the pro-$C^{\ast}$-algebras $\mathcal{A}$ and $\mathcal{B}$, denoted by $\mathcal{A} \otimes \mathcal{B}$, is the completion of the algebraic tensor product $\mathcal{A} \otimes_{\text {alg }} \mathcal{B}$ with respect to the topology determined by a family of $C^{\ast}$-seminorms. Suppose that $\mathcal{X}$ is a Hilbert module over a pro-$C^{\ast}$-algebra $\mathcal{A}$ and $\mathcal{Y}$ is a Hilbert module over a pro-$C^{\ast}$-algebra $\mathcal{B}$. The algebraic tensor product $\mathcal{X} \otimes_{\text {alg }} \mathcal{Y}$ of $\mathcal{X}$ and $\mathcal{Y}$ is a pre-Hilbert $ \mathcal{A} \otimes \mathcal{B}$-module with the action of $ \mathcal{A} \otimes \mathcal{B}$ on $ \mathcal{X} \otimes_{\text {alg }}\mathcal{Y}$ defined by
$$
(\xi \otimes \eta)(a \otimes b)=\xi a \otimes \eta b
$$
for all $ \xi \in \mathcal{X} ,\eta \in \mathcal{Y}, a \in \mathcal{A}$ and $b \in \mathcal{B}$
and the inner product 
$$
\langle\cdot, \cdot\rangle:\left(\mathcal{X} \otimes_{\text {alg }} \mathcal{Y}\right) \times\left(\mathcal{X} \otimes_{\text {alg }} \mathcal{Y}\right) \rightarrow \mathcal{A} \otimes_{\text {alg }}\mathcal{B} 
$$
is  defined by 
$$
\left\langle\xi_{1} \otimes \eta_{1}, \xi_{2} \otimes \eta_{2}\right\rangle=\left\langle\xi_{1}, \xi_{2}\right\rangle \otimes\left\langle\eta_{1}, \eta_{2}\right\rangle.
$$
We also know that for $ z=\sum_{i=1}^{n}\xi_{i}\otimes \eta_{i} $ in $\mathcal{X}\otimes_{alg}\mathcal{Y}$, we have $ \langle z,z\rangle_{\mathcal{A}\otimes\mathcal{B}}=\sum_{i,j}\langle \xi_{i},\xi_{j}\rangle_{\mathcal{A}}\otimes\langle \eta_{i},\eta_{j}\rangle_{\mathcal{B}}\geq0 $ and $ \langle z,z\rangle_{\mathcal{A}\otimes\mathcal{B}}=0 $ if and only if $z=0$.\\
The external tensor product of $\mathcal{X}$ and $\mathcal{Y}$ is the Hilbert module $\mathcal{X} \otimes \mathcal{Y}$ over $\mathcal{A} \otimes \mathcal{B}$ obtained by the completion of the pre-Hilbert $\mathcal{A} \otimes \mathcal{B}$-module $\mathcal{X} \otimes_{\text {alg }} \mathcal{Y}$.

If $P \in M(\mathcal{X})$ and $Q \in M(\mathcal{Y})$ then there is a unique adjointable module morphism $P \otimes Q: \mathcal{A} \otimes  \mathcal{B} \rightarrow \mathcal{X} \otimes \mathcal{Y}$ such that $(P \otimes Q)(a \otimes b)=P(a) \otimes Q(b)$ and $(P \otimes Q)^{*}(a \otimes b)=P^{*}(a) \otimes Q^{*}(b)$ for all $a \in A$ and  all $b \in B$ (see \cite{Joita}).

\begin{theorem}
	Let $\mathcal{X}$ and $\mathcal{Y}$ be two Hilbert pro-$C^{\ast}$-modules over unitary pro-$C^{\ast}$-algebras $\mathcal{A}$ and $\mathcal{B}$, respectively. Let $\{T_{i}\}_{i\in I} $ and  $\{L_{j}\}_{j\in J}$ be two operator frames for $\mathcal{X}$ and $\mathcal{Y}$ with frame operators $S_{T}$ and $S_{L}$ and operator frame bounds $(A,B)$ and $(C,D)$, respectively. Then $\{T_{i}\otimes L_{j}\}_{i\in I,j\in J}  $ is an operator frame for the Hibert $\mathcal{A}\otimes\mathcal{B}$-module $\mathcal{X}\otimes\mathcal{Y}$ with frame operator $ S_{T}\otimes S_{L}$ and lower and upper operator frame bounds $A\otimes C$ and $ B\otimes D $, respectively.
\end{theorem}

\begin{proof}
	By the definition of operator frames $\{T_{i}\}_{i\in I} $ and $\{L_{j}\}_{j\in J}$, we have
	$$A\langle \xi,\xi\rangle\leq\sum_{i\in I}\langle T_{i}\xi,T_{i}\xi\rangle\leq B\langle \xi,\xi\rangle , \forall \xi \in\mathcal{X},$$
	$$C\langle \eta,\eta\rangle\leq\sum_{j\in J}\langle L_{j}\eta,L_{j}\eta\rangle\leq D\langle \eta,\eta\rangle , \forall \eta\in\mathcal{Y}.$$
	Thus
	$$
	\aligned
	(A\langle \xi,\xi\rangle)\otimes (C\langle \eta,\eta\rangle)&\leq\sum_{i\in I}\langle T_{i}\xi,T_{i}\xi\rangle\otimes\sum_{j\in J}\langle L_{j}\eta,L_{j}\eta\rangle\\
	&\leq (B\langle \xi,\xi\rangle)\otimes (D\langle \eta,\eta\rangle) , \forall \xi\in\mathcal{X} ,\forall \eta\in\mathcal{Y}.
	\endaligned
	$$
	So 
	$$
	\aligned
	(A\otimes C)(\langle \xi,\xi\rangle\otimes\langle \eta,\eta \rangle) &\leq\sum_{i\in I,j\in J}\langle T_{i}\xi,T_{i}\xi\rangle\otimes\langle L_{j}\eta,L_{j}\eta\rangle\\
	&\leq (B\otimes D)(\langle \xi,\xi\rangle\otimes\langle \eta,\eta\rangle) , \forall \xi\in\mathcal{X} ,\forall \eta\in\mathcal{Y}.
	\endaligned
	$$
	Consequently, we have
	$$
	\aligned
	(A\otimes C)\langle \xi\otimes \eta,\xi\otimes \eta\rangle &\leq\sum_{i\in I,j\in J}\langle T_{i}\xi\otimes L_{j}\eta,T_{i}\xi\otimes L_{j}\eta\rangle \\
	&\leq (B\otimes D)\langle \xi\otimes \eta,\xi\otimes \eta\rangle, \forall \xi\in\mathcal{X}, \forall \eta\in\mathcal{Y}.
	\endaligned
	$$
	Hence for all $\xi\otimes \eta\in\mathcal{X\otimes Y}$, we have 
	$$
	\aligned
	(A\otimes C)\langle \xi\otimes \eta,\xi\otimes \eta\rangle &\leq\sum_{i\in I,j\in J}\langle(T_{i}\otimes L_{j})(\xi\otimes \eta),(T_{i}\otimes L_{j})(\xi\otimes \eta)\rangle \\
	&\leq (B\otimes D)\langle \xi\otimes \eta,\xi\otimes \eta\rangle.
	\endaligned
	$$
	The last inequality is satisfied for every finite sum of elements in $\mathcal{X}\otimes_{alg}\mathcal{Y}$ and so  it is satisfied for all $z\in\mathcal{X\otimes Y}$. This shows that $\{T_{i}\otimes L_{j}\}_{i\in I,j\in J}  $ is an operator frame for the Hibert $\mathcal{A}\otimes\mathcal{B}$-module $\mathcal{X}\otimes\mathcal{Y}$ with lower and upper operator frame bounds $A\otimes C$ and $ B\otimes D $, respectively.

	By the definition of frame operators $S_{T}$ and $S_{L}$, we have $$S_{T}\xi=\sum_{i\in I}T_{i}^{\ast}T_{i}\xi, \quad \forall \xi\in\mathcal{X},$$
	$$S_{L}\eta=\sum_{j\in J}L_{j}^{\ast}L_{j}\eta,\quad  \forall \eta\in\mathcal{Y}.$$
	Thus 
	$$
	\aligned
	(S_{T}\otimes S_{L})(\xi\otimes \eta)&=S_{T}\xi\otimes S_{L}\eta\\
	&=\sum_{i\in I}T_{i}^{\ast}T_{i}\xi\otimes\sum_{j\in J}L_{j}^{\ast}L_{j}\eta\\
	&=\sum_{i\in I,j\in J}T_{i}^{\ast}T_{i}\xi\otimes L_{j}^{\ast} L_{j}\eta\\
	&=\sum_{i\in I,j\in J}(T_{i}^{\ast}\otimes L_{j}^{\ast})(T_{i}\xi\otimes L_{j}\eta)\\
	&=\sum_{i\in I,j\in J}(T_{i}^{\ast}\otimes L_{j}^{\ast})(T_{i}\otimes L_{j})(\xi\otimes \eta)\\
	&=\sum_{i\in I,j\in J}(T_{i}\otimes L_{j})^{\ast}(T_{i}\otimes L_{j})(\xi\otimes \eta).
	\endaligned
	$$

	Now by the uniqueness of frame operator, the last expression is equal to $S_{T\otimes L}(\xi \otimes \eta)$. Consequently, we have $ (S_{T}\otimes S_{L})(\xi\otimes \eta)=S_{T\otimes L}(\xi\otimes \eta)$. The last equality is satisfied for every finite sum of elements in $\mathcal{X}\otimes_{alg}\mathcal{Y}$ and so  it is satisfied for all $z\in\mathcal{X\otimes Y}$. This  shows that $ (S_{T}\otimes S_{L})(z)=S_{T\otimes L}(z)$. So $S_{T\otimes L}=S_{T}\otimes S_{L}$.
\end{proof}

\section{Dual of operator frame for $Hom_{\mathcal{A}}^{\ast}(\mathcal{X})$}

\begin{definition} 
	Let $T=\{T_{i}\}_{i\in  J}$ be an operator frame for $Hom_{\mathcal{A}}^{\ast}(\mathcal{X})$. A family of operators $\tilde{T}=\{\tilde{T}_{i}\}_{i\in J}$ on $\mathcal{X}$ is called a dual of the operator frame $\{T_{i}\}_{i\in J}$ if they satisfy 
	\begin{equation} \label{108}
	\xi=\sum_{i\in J}T_{i}^{\ast}\tilde{T}_{i}\xi, \quad \forall \xi\in\mathcal{X}.
	\end{equation}
	Furthermore, we call $\{\tilde{T}_{i}\}_{i\in J}$ a dual frame of the operator frame $\{T_{i}\}_{i\in J}$ if $\{\tilde{T}_{i}\}_{i\in J}$ is also an operator frame for $Hom_{\mathcal{A}}^{\ast}(\mathcal{X})$ and satisfies  (\ref{108}).	
\end{definition}

\begin{theorem}
	Every operator frame for $Hom_{\mathcal{A}}^{\ast}(\mathcal{X})$ has a dual frame.
\end{theorem}

\begin{proof}
	If $T=\{T_{i}\}_{i\in J}$ is an operator frame for $Hom_{\mathcal{A}}^{\ast}(\mathcal{X})$ with bounds $A$ and $ B$, then the operator sequence $\tilde{T}=\{T_{i}S_{T}^{-1}\}_{i\in J}$ is a dual frame of $T=\{T_{i}\}_{i\in J}$. In fact, we have
	\begin{displaymath}
	\xi=S_{T}S_{T}^{-1}\xi=\sum_{i\in J}T_{i}^{\ast}T_{i}S_{T}^{-1}\xi=\sum_{i\in J}T_{i}^{\ast}\tilde{T}_{i}\xi, \quad \forall \xi\in\mathcal{X}.
	\end{displaymath}
	And \begin{math}
	\tilde{T}=\{T_{i}S_{T}^{-1}\}_{i\in J}
	\end{math} satisfies
	\begin{displaymath}
	A\|S_{T}\|_{\infty}^{-2}\langle \xi,\xi\rangle\leq\sum_{i\in J}\langle\tilde{T}_{i}\xi,\tilde{T}_{i}\xi\rangle=\sum_{i\in J}\langle\tilde{T}_{i}S_{T}^{-1}\xi,\tilde{T}_{i}S_{T}^{-1}\xi\rangle\leq B\|S_{T}^{-1}\|_{\infty}^{2}\langle \xi,\xi\rangle
	\end{displaymath}
for all $\xi\in\mathcal{X}$.
	Hence \begin{math}
	\{T_{i}S_{T}^{-1}\}_{i\in J}
	\end{math}
	is a  dual frame of $\{T_{i}\}_{i\in J}$., which is called  the canonical dual frame of $\{T_{i}\}_{i\in J}$.
\end{proof}

\begin{remark}
	Assume that $T=\{T_{i}\}_{i\in J}$ is an operator frame for $Hom_{\mathcal{A}}^{\ast}(\mathcal{X})$ with analytic operator $R_{T}$ and  $\tilde{T}=\{\tilde{T}_{i}\}_{i\in J}$ is a dual frame of $T$ with analytic operator $R_{\tilde{T}}$. Then for any $\xi$ in $\mathcal{X}$, we have
	\begin{displaymath}
	\xi=\sum_{i\in J}T_{i}^{\ast}\tilde{T}_{i}\xi=R_{T}^{\ast}R_{\tilde{T}}\xi.
	\end{displaymath}
	This shows that every element of $\mathcal{X}$ can be reconstructed with an operator frame for $Hom_{\mathcal{A}}^{\ast}(\mathcal{X})$ and its dual frame.\\
	Moreover, we also have another fact that for any operator $A$ on $\mathcal{X}$, we get
	\begin{displaymath}
	A\xi=\sum_{i\in J}T_{i}^{\ast}\tilde{T}_{i}A\xi.
	\end{displaymath}
	That is, an association of operator frame and its dual frame can reconstruct pointwisely an  operator on $\mathcal{X}$ and so we can write
	\begin{displaymath}
	A=\sum_{i\in J}T_{i}^{\ast}\tilde{T}_{i}A.
	\end{displaymath} 
\end{remark}

For a frame of submodules $\{W_{i}\}_{i\in J}$ with respect to the family of weights $ \{v_{i}\}_{i\in J} $ for $\mathcal{X}$ with synthesis operator $T_{W,v}$, the sequence $\{u_{i}\}_{i\in J}=\{S_{W,v}^{-1}W_{i}\}_{i\in J}$ is called a  dual frame of $\{W_{i}\}_{i\in J}$ with the operator $S_{W,v}=T_{W,v}T_{W,v}^{\ast}$.

\begin{theorem}
	For a frame of submodules $\{W_{i}\}_{i\in J}$ with respect to the family of weights $ \{v_{i}\}_{i\in J} $ for $\mathcal{X}$, define $T_{i}=v_{i}S_{W,v}\pi_{W_{i}}S_{W,v}^{-1}$ and $Q_{i}=v_{i}\pi_{u_{i}}S_{W,v}^{-1}$ such that $S_{W,v}$  and its inverse are uniformly bounded. Then $Q=\{Q_{i}\}_{i\in J}$ and $T=\{T_{i}\}_{i\in J}$ are all operator frames for $Hom_{\mathcal{A}}^{\ast}(\mathcal{X})$ and $Q$ is a dual frame of $T$.
\end{theorem}

\begin{proof}
	Assume that $\{W_{i}\}_{i\in J}$ has frame bounds $A$ and $ B$.
	\begin{itemize}
		\item [(1)] $T=\{T_{i}\}_{i\in J}$ is an operator frame for $Hom_{\mathcal{A}}^{\ast}(\mathcal{X})$.\\
		For any $\xi\in\mathcal{X}$, we have
		\begin{align*}
		\sum_{i\in J}\langle T_{i}\xi,T_{i}\xi\rangle&=\sum_{i\in J}\langle v_{i}S_{W,v}\pi_{W_{i}}S_{W,v}^{-1}\xi,v_{i}S_{W,v}\pi_{W_{i}}S_{W,v}^{-1}\xi\rangle\\
		&\leq\|S_{W,v}\|_{\infty}^{2}B\langle S_{W,v}^{-1}\xi,S_{W,v}^{-1}\xi\rangle
		\\&\leq B\|S_{W,v}\|^{2}\|_{\infty}S_{W,v}^{-1}\|_{\infty}^{2}\langle \xi,\xi\rangle.	
		\intertext{On the other hand,}	
		\sum_{i\in J}\langle T_{i}\xi,T_{i}\xi\rangle&=\sum_{i\in J}\langle v_{i}S_{W,v}\pi_{W_{i}}S_{W,v}^{-1}\xi,v_{i}S_{W,v}\pi_{W_{i}}S_{W,v}^{-1}\xi\rangle\\
		&\geq\|S_{W,v}^{-1}\|_{\infty}^{-2}A\langle S_{W,v}^{-1}\xi,S_{W,v}^{-1}\xi\rangle\\
		&\geq A\|S_{W,v}^{-1}\|_{\infty}^{-2}\|S_{W,v}\|_{\infty}^{-2}\langle \xi,\xi\rangle.
		\end{align*}
		Thus $T=\{T_{i}\}_{i\in J}$ is an operator frame.
		\item [(2)] $Q=\{Q_{i}\}_{i\in J}$ is also an operator frame for $Hom_{\mathcal{A}}^{\ast}(\mathcal{X})$. The proof is similar to $(1)$.
		\item [(3)] If $T_{U,v}=S_{W,v}^{-1}T_{W,v}S_{W,v}$, then $T_{U,v}^{\ast}=S_{W,v}^{-1}T_{W,v}^{\ast}S_{W,v}$ and
		$S_{U,v}=S_{W,v}$.\\
		It is easy to check that $\pi_{u_{i}}=S_{W,v}^{-1}\pi_{W_{i}}S_{W,v}$.\\
		Thus $T_{U,v}=S_{W,v}^{-1}T_{W,v}S_{W,v}, T_{U,v}^{\ast}=S_{W,v}^{-1}T_{W,v}^{\ast}S_{W,v}$ and so 
		\begin{align*}
		S_{U,v}&=T_{U,v}T_{U,v}^{\ast}\\
		&=S_{W,v}^{-1}T_{W,v}S_{W,v}S_{W,v}^{-1}T_{W,v}^{\ast}S_{W,v}\\
		&=S_{W,v}^{-1}T_{W,v}T_{W,v}^{\ast}S_{W,v}\\
		&=S_{W,v}^{-1}S_{W,v}S_{W,v}\\
		&=S_{W,v}.
		\intertext{Hence, for any $\xi\in\mathcal{X}$, we compute}
		\sum_{i\in J}T_{i}^{\ast}Q_{i}\xi&=\sum_{i\in J}v_{i}S_{W,v}^{-1}\pi_{W_{i}}S_{W,v}v_{i}S_{W,v}^{-1}\pi_{W_{i}}S_{W,v}S_{W,v}^{-1}x\\
		&=S_{W,v}^{-1}(\sum_{i\in  J}v_{i}^{2}\pi_{W_{i}}\xi)\\
		&=S_{W,v}^{-1}S_{W,v}\xi\\
		&=\xi.
		\end{align*}
		This shows that $Q=\{Q_{i}\}_{i\in J}$ is an dual operator frame of the operator frame $T=\{T_{i}\}_{i\in J}$.
	\end{itemize}
This completes the proof.
\end{proof}

\begin{theorem}
	Assume that $\{W_{i}\}_{i\in J}$ is a Parseval frame of submodules for a Hilbert $\mathcal{A}$-module $\mathcal{X}$. Then $\{v_{i}\pi_{W_{i}}\}_{i\in J}$ is an operator frame for $Hom_{\mathcal{A}}^{\ast}(\mathcal{X})$ and a dual frame of itself. 
\end{theorem}
\begin{proof}
	When $\{W_{i}\}_{i\in J}$ is a Parseval frame of submodules, $S_{W,v}=I$. Clearly, this theorem is a consequence of the last theorem. 
\end{proof}

\begin{theorem}	
	Let $\{T_{i}\}_{i\in I} $ and $\{R_{j}\}_{j\in J} $ be  two operator frames respectively in $Hom_{\mathcal{A}}^{\ast}(\mathcal{X})$ and $Hom_{\mathcal{A}}^{\ast}(\mathcal{Y})$, with duals $\{\tilde{T_{i}}\}_{i\in I}$ and  $\{\tilde{R_{j}}\}_{j\in J} $ respectively. Then $\{\tilde{T_{i}}\otimes\tilde{R_{j}}\}_{i,j\in I,J}$ is a dual of $\{T_{i}\otimes R_{j}\}_{i,j\in I,J}$.
\end{theorem}

\begin{proof}	
	By definition, for all$ x \in \mathcal{X}$ and $ y \in \mathcal{Y}$, we have
	\begin{equation*}
	x=\sum_{i\in I}T_{i}^{\ast}\tilde{T_{i}}x  , 
	\end{equation*}
	\begin{equation*}
	y=\sum_{j\in J}R_{j}^{\ast}\tilde{R_{j}}y  
	\end{equation*}
and so  
	\begin{equation*}
	(x\otimes y)=  \sum_{i\in I}T_{i}^{\ast}\tilde{T_{i}}x\otimes\sum_{j\in J}R_{j}^{\ast}\tilde{R_{j}}y,
	\end{equation*}
	\begin{equation*}
	\sum_{i\in I}T_{i}^{\ast}\tilde{T_{i}}x\otimes\sum_{j\in J}R_{j}^{\ast}\tilde{R_{j}}y=\sum_{i,j\in I,J}T_{i}^{\ast}\tilde{T_{i}}x\otimes R_{j}^{\ast}\tilde{R_{j}}y,
	\end{equation*}	
	\begin{equation*}	
	\sum_{i\in I}T_{i}^{\ast}\tilde{T_{i}}x\otimes\sum_{j\in J}R_{j}^{\ast}\tilde{R_{j}}y=\sum_{i,j\in I,J}(T_{i}^{\ast}\otimes R_{j}^{\ast}).(\tilde{T_{i}}x\otimes\tilde{R_{j}}y),
	\end{equation*}
	\begin{equation*}
	\sum_{i\in I}T_{i}^{\ast}\tilde{T_{i}}x\otimes\sum_{j\in J}R_{j}^{\ast}\tilde{R_{j}}y=\sum_{i,j\in I,J}(T_{i}\otimes R_{j})^{\ast}.(\tilde{T_{i}}\otimes\tilde{R_{j}}).(x\otimes y).
	\end{equation*}
Thus  $\{\tilde{T_{i}}\otimes\tilde{R_{j}}\}_{i,j \in I,J}$ is a dual of $\{T_{i}\otimes R_{j}\}_{i,j \in I,J}$.
\end{proof}

\begin{corollary}	
	Let $\{T_{i,j}\}_{0\leq i\leq n; j\in J}$ be a family of operator frames with  their duals $\{\tilde{T}_{i,j}\}_{0\leq i\leq n; j\in J}$  for  $ 0 \leq i \leq n $.  Then $\{\tilde{T}_{0,j}\otimes \tilde{T}_{1,j}\otimes......\otimes\tilde{T}_{n,j}\}_{j\in J}$ is a dual of $\{T_{0,j}\otimes T_{1,j}\otimes......\otimes T_{n,j}\}_{j\in J}$.
\end{corollary}

\bigskip

\section*{Declarations}

\medskip

\noindent \textbf{Availablity of data and materials}\newline
\noindent Not applicable.

\medskip

\noindent \textbf{Competing  interest}\newline
\noindent The authors declare that they have no competing interests.

\medskip

\noindent \textbf{Fundings}\newline
\noindent    The authors declare that there is no funding available fot this paper.

\medskip

\noindent \textbf{Authors' contributions}\newline
\noindent The authors equally conceived of the study, participated in its
design and coordination, drafted the manuscript, participated in the
sequence alignment, and read and approved the final manuscript.

\end{document}